\documentclass[11pt]{amsart}
\usepackage[letterpaper,margin=1in]{geometry} 
\usepackage[colorlinks,
             linkcolor=blue,
             citecolor=green,
             pdfproducer={pdfLaTeX},
             pdfpagemode=None,
             bookmarksopen=true
             bookmarksnumbered=true]{hyperref}
\usepackage{parskip}
\usepackage[matrix,arrow,ps,color,line,curve,frame,all]{xy} 

\title{Residually finite quantum group algebras}

\author{Alexandru Chirvasitu}
\thanks{Email address: \MAILTO{chirva@uw.edu}}



\usepackage{amsmath,amsthm,stmaryrd}
\usepackage{cleveref}

\newtheorem*{ntheorem}{Theorem}

\newtheorem{lemma}{Lemma}[section]
\newtheorem{theorem}[lemma]{Theorem}
\newtheorem{proposition}[lemma]{Proposition}
\newtheorem{corollary}[lemma]{Corollary}

\theoremstyle{definition} 

\newtheorem{definition}[lemma]{Definition}
\newtheorem{examplenodiamond}[lemma]{Example}
\newtheorem{remarknodiamond}[lemma]{Remark}

\newenvironment{example}{\begin{examplenodiamond}}{\hfill\ensuremath\blacklozenge\end{examplenodiamond}}
\newenvironment{remark}{\begin{remarknodiamond}}{\hfill\ensuremath\blacklozenge\end{remarknodiamond}}

\renewcommand\qedhere{\qed}

\newcounter{stepofproof}

\crefname{section}{Section}{Section}
\crefformat{section}{#2Section~#1#3} 
\Crefformat{section}{#2Section~#1#3} 

\crefname{subsection}{}{Subsections}
\crefformat{subsection}{\S#2#1#3} 
\Crefformat{subsection}{\S#2#1#3}

\crefname{definition}{Definition}{Definitions}
\crefformat{definition}{#2Definition~#1#3} 
\Crefformat{definition}{#2Definition~#1#3} 

\crefname{example}{Example}{Examples}
\crefformat{example}{#2Example~#1#3} 
\Crefformat{example}{#2Example~#1#3} 

\crefname{examplenodiamond}{Example}{Examples}
\crefformat{examplenodiamond}{#2Example~#1#3} 
\Crefformat{examplenodiamond}{#2Example~#1#3} 

\crefname{remark}{Remark}{Remarks}
\crefformat{remark}{#2Remark~#1#3} 
\Crefformat{remark}{#2Remark~#1#3} 

\crefname{remarknodiamond}{Remark}{Remarks}
\crefformat{remarknodiamond}{#2Remark~#1#3} 
\Crefformat{remarknodiamond}{#2Remark~#1#3} 

\crefname{convention}{Convention}{Conventions}
\crefformat{convention}{#2Convention~#1#3} 
\Crefformat{convention}{#2Convention~#1#3} 

\crefname{lemma}{Lemma}{Lemmas}
\crefformat{lemma}{#2Lemma~#1#3} 
\Crefformat{lemma}{#2Lemma~#1#3} 

\crefname{proposition}{Proposition}{Propositions}
\crefformat{proposition}{#2Proposition~#1#3} 
\Crefformat{proposition}{#2Proposition~#1#3} 

\crefname{corollary}{Corollary}{Corollaries}
\crefformat{corollary}{#2Corollary~#1#3} 
\Crefformat{corollary}{#2Corollary~#1#3} 

\crefname{theorem}{Theorem}{Theorems}
\crefformat{theorem}{#2Theorem~#1#3} 
\Crefformat{theorem}{#2Theorem~#1#3} 

\crefname{assumption}{Assumption}{Assumptions}
\crefformat{assumption}{#2Assumption~#1#3} 
\Crefformat{assumption}{#2Assumption~#1#3} 

\crefname{equation}{}{}
\crefformat{equation}{(#2#1#3)} 
\Crefformat{equation}{(#2#1#3)}

\crefname{proofstep}{Step}{Steps}
\crefformat{proofstep}{#2Step~#1#3} 
\Crefformat{proofstep}{#2Step~#1#3}

\usepackage{tikz}
\usetikzlibrary{arrows,decorations.pathreplacing,decorations.markings,shapes.geometric,through,fit,shapes.symbols,positioning,decorations.pathmorphing}

\tikzset{
    state/.style={
           rectangle,
           rounded corners,
           draw=black, 
           minimum height=2em,
           inner sep=4pt,
           text centered,
           },
}

\newcommand\cat[1]{\textsc{#1}}
\newcommand\define[1]{\emph{#1}}

\newcommand\arXiv[1]{\href{http://arxiv.org/abs/#1}{\nolinkurl{arXiv:#1}}}
\newcommand\MRnumber[1]{\href{http://www.ams.org/mathscinet-getitem?mr=#1}{\nolinkurl{MR#1}}}
\newcommand\DOI[1]{\href{http://dx.doi.org/#1}{\nolinkurl{DOI:#1}}}
\newcommand\MAILTO[1]{\href{mailto:#1}{\nolinkurl{#1}}}

\usepackage{amsfonts,amssymb}

\newcommand\bC{\mathbb C}

\newcommand\bP{\mathbb P}

\newcommand\bR{\mathbb R}

\newcommand\bZ{\mathbb Z}

\newcommand\cC{\mathcal C}

\newcommand\cM{\mathcal M}

\renewcommand\lim{\varprojlim}

\newcommand\id{\mathrm{id}}
\newcommand\ol{\overline}

\begin{document}

\maketitle

\begin{abstract}
We show that provided $n\ne 3$, the involutive Hopf $*$-algebra $A_u(n)$ coacting universally on an $n$-dimensional Hilbert space has enough finite-dimensional representations in the sense that every non-zero element acts non-trivially in some finite-dimensional $*$-representation. This implies that the discrete quantum group with group algebra $A_u(n)$ is maximal almost periodic (i.e. it embeds in its quantum Bohr compactification), answering a question posed by P. So\l tan in \cite{MR2210362}. 

We also prove analogous results for the involutive Hopf $*$-algebra $B_u(n)$ coacting universally on an $n$-dimensional Hilbert space equipped with a non-degenerate bilinear form. 
\end{abstract}

\section{Introduction}\label{se.intro}

This paper is concerned with the property of \define{residual finite-dimensionality} (or RFD for short) for operator algebras associated to discrete quantum groups. For a $C^*$-algebra $A$ being RFD means having a separating family representations on finite-dimensional Hilbert spaces; in other words, for any $0\ne a\in A$ there is a $C^*$-algebra homomorphism $\pi:A\to M_n(\bC)$ that does not annihilate $a$. 

The RFD property is in a sense the exact opposite of simplicity: A $C^*$-algebra is simple if it has no non-obvious quotients, whereas it is RFD if it has plenty of small quotients. The free group $F_2$ on two generators is the perfect example for both extremes: It is a standard result that its so-called \define{reduced} $C^*$-algebra (i.e. the closure in $B(\ell^2(F_2))$ of the left regular representation) is simple, while the $C^*$-algebra universally generated by two unitaries (the \define{full} $C^*$-algebra of $F_2$) is RFD \cite{MR590864}. 

In fact freeness of some sort has been central in investigating the RFD property. In \cite{MR1052340} for instance, the authors prove among other things that the full $C^*$-algebra of a monoid that is a coproduct (in the category of monoids) of free groups and free or finite monoids is RFD; and in \cite{MR1168356} Exel and Loring show that the coproduct of two RFD $C^*$-algebras is again RFD, generalizing all of the previously-cited results (the full $C^*$-algebra $C^*(F_2)$ of $F_2$ for instance is the coproduct of two copies of the RFD full $C^*$-algebra $C^*(\bZ)$). 

All of the above references deal extensively with group $C^*$-algebras of discrete groups. The same kinds of issues are raised in \cite{MR2210362} in the context of discrete \define{quantum} groups, where the main objects under consideration are the so-called CQG algebras of \cite{MR1310296}. 

We recall below (\Cref{subse.CQG}) that these are algebras (in fact Hopf algebras) which should be thought of as comprising the representative functions on a compact ``quantum group''. By a kind of non-commutative Pontryagin duality, such an algebra is also trying to be the group algebra of a discrete quantum group. Just as for a classical discrete group, a CQG algebra has two extremal completions to a $C^*$-algebra: a largest one called `full' and a smallest one called `reduced'. 

The paper \cite{MR2210362} defines and constructs Bohr compactifications for discrete quantum groups. The procedure is parallel to the classical one of compactifying an ordinary discrete group, and many of the problems one can pose classically make sense here too. In particular, there is a notion of \define{maximal almost periodicity} for a discrete quantum group, meaning morally that it embeds in its Bohr compactification (see the discussion at the very end of the paper for some more details). It is here that the RFD property becomes relevant: So\l tan shows in \cite[4.10 (1)]{MR2210362} that a discrete quantum group is indeed maximal almost periodic provided the full $C^*$ completion of the underlying CQG algebra is RFD.

In view of all of the above, the CQG algebras $A_u(Q)$ from \Cref{ex.free} are now natural candidates for testing RFD-ness, since they are in a sense ``universal''; as explained in \Cref{ex.free}, this means that the family of all $A_u(Q)$ is for compact quantum groups what function algebras of unitary groups are for ordinary compact groups. Moreover, as discrete quantum groups the $A_u(Q)$ can be regarded as analogues of the free groups (for instance because their representation rings are non-commutative polynomial rings; see \cite{MR1484551} and the discussion below, in \Cref{se.main}). 

It turns out that $A_u(Q)$ cannot possibly be RFD unless $Q$ is scalar, so that we may as well assume it is the identity matrix $I_n$ for some $n$; we denote $A_u(I_n)$ by $A_u(n)$. The question of whether or not the full $C^*$ completions of $A_u(n)$ are RFD is then posed explicitly in \cite{MR2210362}. 

Here we prove somewhat less than this, but still enough to get maximal almost periodicity. The main observation is that the latter property does not require that the full $C^*$ envelope of the CQG algebra be RFD; instead, it is enough that the CQG algebra have some RFD $C^*$ completion. Equivalently, this means that the CQG algebra itself is RFD in the obvious sense: It has a separating family of finite-dimensional $*$-representations (see \Cref{def.RFD}). It is this purely algebraic formulation of RFD-ness that is central to the paper and its main result (\Cref{th.main}):

\begin{ntheorem}
The CQG algebras $A_u(n)$ are RFD provided $n\ne 3$, as are the CQG algebras $B_u(n)=B_u(I_n)$ coacting universally on an $n$-dimensional Hilbert space endowed with a bilinear form from see \Cref{ex.ortho}. \qedhere
\end{ntheorem}

Recalling from \cite{MR1484551} that the reduced $C^*$ completion of $A_u(n)$ is simple, this shows that the ``group algebra'' $A_u(n)$ exhibits the same wide range of behaviors as the group algebra of an ordinary free group: Its small $C^*$ completions are simple, but there are larger ones that are RFD.

In addition to the motivation coming from maximal almost periodicity, there is a second strand of ideas that is very much in the spirit of this paper. The following finiteness property related to RFD-ness was introduced relatively recently in \cite{MR2679923}, and studied further e.g. in \cite{MR2681256,MR2914541}:

\begin{definition}
A Hopf algebra is said to be \define{inner linear} if it has a finite-codimensional ideal containing no non-trivial Hopf ideal. 
\end{definition}

In other words, it is supposed to have a finite-dimensional representation that does not factor through any proper Hopf algebra quotient. Moreover, there is a version appropriate for $*$-structures (\cite[5.1]{MR2681256}):

\begin{definition}
A complex Hopf $*$-algebra is \define{inner unitary} if it has a $*$-representation on a finite-dimensional Hilbert space whose kernel does not contain a non-zero Hopf $*$-ideal.  
\end{definition}

Inner unitarity is stronger than RFD-ness, as will become clear from the discussion below. Indeed, every $*$-algebra has a canonical RFD quotient, and according to \Cref{pr.RFD_quot_CQG} the RFD quotient of a CQG algebra is a CQG quotient. But then for a CQG algebra that is not RFD every finite-dimensional $*$-representation on a Hilbert space factors through the proper RFD quotient. Hence, there can be no representation exhibiting inner unitarity.

We refer the reader to the cited papers for further information on these notions. We do not prove inner unitarity or linearity for any of the CQG algebras under consideration here, but there are clearly relationships between these concepts that are worth noting.

The summary of the paper is as follows:

In the next section we collect some of the auxiliary material we need, mostly on the RFD property for $*$-algebras in general and for CQG algebras in particular. Most helpfully, it turns out that the coalgebra structure of a CQG algebra plays an important role in investigating RFD-ness (\Cref{cor.RFD_quot_CQG}). This means that in actually proving residual finite-dimensionality for a CQG algebra, we can do computations inside the category of comodules and hence avail ourselves of all of the extra structure that comes with having a comultiplication. I hope some of the material here will be of some independent interest.  

\Cref{se.main} contains the main result cited above, \Cref{th.main}. In addition, the passage from the RFD-ness of $A_u(n)$ to $B_u(n)$ is made through a more general result linking RFD-ness of a compact quantum group to that of a ``quotient group'' (\Cref{th.lift_RFD}). 

The end of \Cref{se.main} and the paper consists of a brief discussion of how all of this applies to maximal almost periodicity, in slightly more detail than we sketched above.

\section{Preliminaries}\label{se.prel}

All algebraic objects appearing below (algebras, coalgebras, Hopf algebras, etc.) live over the field $\bC$. All algebras and coalgebras are unital and respectively counital.

For background on coalgebra, comodule and Hopf algebra theory we refer to \cite{MR0252485,MR1243637,MR2894855}.

We denote comultiplications and counits of coalgebras by $\Delta$ and $\varepsilon$ respectively, perhaps with a subscript when wishing to indicate which coalgebra is being talked about (e.g. $\Delta_C$). Antipodes of Hopf algebras are usually denoted by $S$, with the same option of adding a subscript ($S_H$ for the antipode of $H$). We use Sweedler notation for comultiplication: $\Delta(h)=h_1\otimes h_2$

Unless specified otherwise, comodules are right and modules are left; the symbol $\cM^C$ stands for the category of $C$-comodules.

\subsection{CQG algebras}\label{subse.CQG}

The main references are \cite{MR1310296}, where the notion was first introduced, \cite[Sections 11.3, 11.4]{MR1492989} and the survey \cite{MR1741102}. We do not need to recall the concept in great detail. For our purposes, it is enough to remember that a CQG algebra is a Hopf $*$-algebra with an additional technical property ensuring that its comodules admit inner products invariant under the coaction in some sense that we will not make precise. 

Recall also that a Hopf $*$-algebra is a Hopf algebra as well as a $*$-algebra (i.e. it is equipped with an involutive, conjugate-linear, multiplication reversing self-map $*$) such that the comultiplication and the counit are $*$-algebra homomorphisms.

The condition we have not spelled out is meant to ensure that these objects behave in many ways like algebras of representative functions on compact groups\footnote{representative functions are linear spans of matrix coefficients coming from finite-dimensional representations of the compact group}. For this reason we also refer to a CQG algebra as a \define{compact quantum group}. 
 
Algebras of representative functions on ordinary compact groups provide examples, as do group algebras of discrete groups. On the other hand, the objects we work with below are

\begin{example}\label{ex.free}
For an $n\times n$ positive self-adjoint matrix $Q$, let $A_u(Q)$ be the $*$-algebra freely generated by the $n^2$ elements $u_{ij}$, $i,j=\overline{1,n}$ such the that $u=(u_{ij})$ and $Q^{\frac 12}\overline uQ^{-\frac 12}$ are both unitary as elements of $M_n(A_u(Q))$ (where $\overline u=(u_{ij}^*)$). 

$A_u(Q)$ can be made into a CQG algebra by declaring that $u_{ij}$ are the usual $n^2$ basis elements of a matrix coalgebra in the sense that \[ \Delta(u_{ij})=\sum_k u_{ik}\otimes u_{kj},\quad \varepsilon(u_{ij})=\delta_{ij} \] (a theme that will come up again and again in these examples; they are all defined by imposing relations on the $n^2$ matrix counits of an $n\times n$ matrix coalgebra). 

The algebras $A_u(Q)$ were introduced by Wang and van Daele in \cite{MR1382726}, and they are the quantum analogues of unitary groups: Every finitely generated CQG algebra is a quotient of one of them, meaning, in dual language, that every ``compact quantum Lie group'' embeds in the compact quantum group associated to $A_u(Q)$ for some $Q$.
%
%
\end{example}

\begin{example}\label{ex.ortho}
Now let $Q$ be a matrix with the property $Q\overline{Q}\in \bR I_n$ (just like the positivity assumption of the previous example, this avoids some redundance and degeneracy). The CQG algebra $B_u(Q)$ was also introduced in \cite{MR1382726} (denoted there by $A_o(Q)$) as the $*$-algbra obtained from $A_u(Q)$ by imposing the additional assumption that all $u_{ij}$ be self-adjoint. 
%
\end{example}

More specifically we focus on $A_u(I_n)$ and $B_u(I_n)$, which we denote by $A_u(n)$ and respectively $B_u(n)$ (or on occasion by $A$ and $B$).

Rephrasing \Cref{ex.free} slightly, $A=A_u(n)$ is the $*$-algebra freely generated by $n^2$ elements $u_{ij}$ subject to the constraints that both $(u_{ij})_{i,j}$ and $(u_{ji})_{i,j}$ be unitary elements of $M_n(A)$. The same goes for $B_u(n)$, except that we denote the generators by $v_{ij}$to avoid confusion, and we have the additional relations $v_{ij}^*=v_{ij}$.  On occasion, we refer to $A_u(n)$ and $B_u(n)$ as the \define{free unitary} and respectively \define{free orthogonal} compact quantum groups.  

The reason for specializing to $Q=I_n$ will become apparent below (see the discussion immediately preceding \Cref{se.main}). Briefly, the results we prove in the next two sections do not stand a chance of being true for CQG algebras whose antipodes do not square to the identity.

Just as algebras of representative functions separate points of compact groups and hence sit densely inside their $C^*$ completions, so too all CQG algebras can be completed to $C^*$-algebras (usually in more than one way). In fact, compact quantum groups appeared historically as $C^*$-algebras equipped with additional structure (\cite{MR901157,MR943923}), and only afterwards in their purely algebraic guise. 

Such completions will come up on occasion; the references cited in this section will do nicely for background on the more analytic aspects.

\subsection{Residually finite-dimensional \texorpdfstring{$*$}{*}-algebras and finite duals}

An algebra $A$ is usually said to be residually finite-dimensional if for any non-zero $a\in A$ there is some finite-dimensional module of $A$ which is not annihilated by $a$. We then also say that the finite-dimensional representations of $A$ \define{separate} the elements of $A$, or that they form a \define{separating family}. 

We are interested here in a modified version of this. Recall from the introduction that the main objects of study are $*$-algebras, and one tries to show that they have a separating family of $*$-representations on finite-dimensional Hilbert spaces. This justifies changing the standard term slightly to suit the present setting.

\begin{definition}\label{def.RFD}
A $*$-algebra is \define{residually finite-dimensional} or RFD if for any $0\ne a\in A$ there is some $*$-prepresentation $\pi$ of $A$ on a Hilbert space such that $\pi(a)\ne 0$.

The \define{RFD quotient} $\ol{A}$ of an arbitrary $*$-algebra $A$ is the quotient of $A$ by the intersection of the kernels of all $*$-homomorphisms $A\to M_n(\bC)$.   
\end{definition}

In other words, $\ol{A}$ is the largest quotient of $A$ which is RFD. It will be useful later on to give an alternate description of the RFD quotient as the image of a canonical map from $A$ into a kind of ``double dual''.

\begin{remark}\label{rem.comm_RFD}
Note that any commutative $*$-algebra $A$ that embeds in some $C^*$-algebra $B$ is automatically RFD. Indeed, the $C^*$-algebra homomorphisms from the closure of $A$ in $B$ to $\bC$ form a separating family of $1$-dimensional $*$-representations for $A$.  
\end{remark}

First, we recall the notation $A^\circ$ for the so-called finite dual of $A$ from \cite[Chapter VI]{MR0252485}. This is the subset of the full dual vector space $A^*$ consisting of $f\in A^*$ which vanish on a finite-codimensional ideal of $A$. Another way to phrase this is as follows: Consider all finite-dimensional representations $\pi:A\to M_n(\bC)$ (for all possible $n$). Composing further with the $n^2$ matrix entries $M_n(\bC)\to\bC$ we get functionals on $A$. Then, $A^\circ$ is the linear span of all of these functionals. 

The motivation for the introduction of $A^\circ$ in \cite{MR0252485} comes from the fact that it is the largest subspace of $A^*$ which can be endowed with a coalgebra structure via the dual $\Delta:A^*\to (A\otimes A)^*$ of the multiplication map $A\otimes A\to A$. This means that $A^\circ$ is the preimage of $A^*\otimes A^*\le (A\otimes A)^*$ in $A^*$, and $\Delta(A^\circ)$ is contained in $A^\circ\otimes A^\circ$. As a consequence, $A^\circ$ is always a coalgebra. Moreover, the contravariant functors $*:\cat{Coalgebras}\to\cat{Algebras}$ (full dual) and $\circ:\cat{algebras}\to\cat{Coalgebras}$ are adjoint on the right: There is a natural bijection
\[
	\text{coalgebra maps }C\to A^\circ\quad \cong\quad \text{algebra maps }A\to C^*. 
\]
In particular, there is a canonical algebra map $A\to A^{\circ *}$, which is simply the composition $A\to A^{**}\to A^{\circ *}$. 

Once more, it will be convenient to borrow the notation but change the meaning of `$\circ$' slightly to suit us better in the context of $*$-algebras.

\begin{definition}
For a $*$-algebra $A$, the \define{finite dual} $A^\circ$ is the linear span in $A^*$ of all matrix coefficients for all $*$-homomorphisms $A\to M_n(\bC)$. 
\end{definition}

Note that $A^\circ$ is what we will henceforth call a \define{$*$-coalgebra}: It admits an involutive conjugate linear map $*:A^\circ\to A^\circ$ which preserves the counit and reverses the comultiplication, defined by 
\begin{equation}\label{eq.dual_functional}
	f^*(a)=f(a^*)^*,\quad a\in A,\ f\in A^\circ,
\end{equation}
the outer star being complex conjugation of a number. The full dual of a $*$-coalgebra is a $*$-algebra with the $*$-structure \Cref{eq.dual_functional} again, and we will leave it to the reader to check that the composition $A\to A^{**}\to A^{\circ *}$ is a morphism of $*$-algebras. The connection with the previous discussion is as follows:

\begin{proposition}\label{pr.RFD_quot_alg}
For any $*$-algebra $A$, the RFD quotient $\ol{A}$ is the image of the canonical map $A\to A^{\circ *}$. 
\end{proposition}
\begin{proof}
This is tautological: An element of $A$ is annihilated by all $*$-homomorphisms $A\to M_n(\bC)$ if and only if it is annihilated by all $f\in A^\circ$, if and only if it maps to zero under $A\to A^{\circ *}$. 
\end{proof}

Below, we will also need to look into whether certain free products of $*$-algebras are RFD. To that end, we finish this subsection with the following result.

\begin{proposition}\label{pr.RFD_free_prod}
Let $A$ and $B$ be RFD $*$-algebras. Then, the coproduct $A*B$ in the category of $*$-algebras is RFD. 
\end{proposition}

Before going into the proof, let us note that an analogous statement holds for $C^*$-algebras if one defines the RFD property as above, but considering only continuous homomorphisms $A\to M_n(\bC)$. The category of $C^*$-algebras has coproducts, and then, according to \cite[Theorem 3.2]{MR1168356}, \Cref{pr.RFD_free_prod} holds verbatim for $C^*$-algebras $A$ and $B$. We will use that result in the proof of \Cref{pr.RFD_free_prod}, but indirectly. For one thing, arbitrary $*$-algebras need not have enveloping $C^*$-algebras, and so it is unclear how to immediately reduce the proposition to its analytic version; for another, it is not clear to me whether the enveloping $C^*$-algebra of an RFD $*$-algebra, if it exists, is again RFD in the sense of Exel and Loring.

\begin{proof}[Proof of \Cref{pr.RFD_free_prod}]
Consider the sets $I$ and $J$ of surjections $\pi_i:A\to A_i$ and $\pi_j:B\to B_j$ respectively. We can partially order $I$ by $\pi_i\le\pi_{i'}$ if $\pi_{i'}$ factors through $\pi_i$, and similarly for $J$.
\[
		\tikz[anchor=base]{
  		\path (0,0) node (A) {$A$} +(2,0) node (Ai) {$A_i$} +(4,0) node (Ai') {$A_{i'}$};
  		\draw[->] (A) -- (Ai) node[pos=.5,auto] {$\pi_i$};
  		\draw[->] (A) .. controls (1,-.5) and (3,-.5) .. (Ai') node[pos=.5,auto,swap] {$\pi_{i'}$};
  		\draw[->,dashed] (Ai) -- (Ai');
  	}
\]
Note that $I$ and $J$ are both \define{filtered}: For any two elements in $I$ (or $J$) there is an element in $I$ (respectively $J$) less than or equal to both. Indeed, if $\pi_i$ and $\pi_{i'}$ are elements of $I$, then the quotient modulo the intersection $\ker(\pi_i)\cap\ker(\pi_{i'})$ is smaller than either of them. 

Regarding the posets $I$ and $J$ as categories in the usual way, with an arrow for each $\le$ relation between two elements, $i\mapsto A_i$ and $j\mapsto B_j$ are functors from $I$ and $J$ respectively to the category of $*$-algebras. The RFD hypothesis means simply that the resulting morphisms $A\to\varprojlim_I A_i$ and $B\to\varprojlim_JB_j$ are one-to one, where $\varprojlim$ means limit in the category of $*$-algebras. 

Since coproducts of embeddings in the category of algebras are again embeddings, the canonical map $A*B\to \left(\varprojlim_I A_i\right)*\left(\varprojlim_J B_j\right)$ is one-to-one. The right hand side is canonically isomorphic to $\varprojlim_{I\times J}(A_i*B_j)$, where $I\times J$ is the product poset with $(i,j)\le (i',j')$ iff $i\le i'$ and $j\le j'$. This follows from the fact that $I$ and $J$ are filtered as noted above, and filtered limits commute with finite colimits in any category where they make sense (this is the categorical dual to \cite[Theorem IX.2.1]{MR1712872}). 

In conclusion, we are embedding the $*$-algebra $A*B$ into $\varprojlim_{I\times J}(A_i*B_j)$. We would be done if we knew that $A_i*B_j$ are RFD, but now we can make use of the Exel-Loring result cited above: Being finite-dimensional $C^*$-algebras, $A_i$ and $B_j$ are RFD in the $C^*$ sense, and hence \cite[Theorem 3.2]{MR1168356} applies to them. It implies that the $C^*$-completion is RFD in the $C^*$ sense, and it is easy to see in this case that the $*$-algebra $A_i*B_j$ embeds in its $C^*$-completion.   
\end{proof}

\begin{remark}
\Cref{pr.RFD_free_prod} is purely algebraic, and it should have a purely algebraic proof. It does, but going through \cite{MR1168356} provided a shorter argument. 
\end{remark}

\subsection{Residually finite-dimensional CQG algebras and Hopf duals}

Keeping in mind our goal of eventually proving that $A_u(n)$ and $B_u(n)$ are RFD, we specialize some of the discussion above to the case of CQG algebras.  

Recall from the previous subsection the adjoint contravariant functors between algebras and coalgebras implemented by taking duals and finite duals. One might think that in the context of the present paper, where we have modified `$\circ$' to take into account $*$-structures, the analogous result holds: $*$ and our version of $\circ$ implement an adjunction on the right between $*$-algebras and $*$-coalgebras. This is not true! The problem is that $\circ$ has to do with mapping not into arbitrary finite-dimensional $*$-algebras, but rather into finite-dimensional $C^*$-algebras. As a consequence, the corresponding category of $*$-coalgebras that will make the adjunction work is smaller. Instead of spelling out how that works, we consider straight away only the case when everything in sight is a CQG algebra, and is hence both an algebra and a coalgebra.

Recall that the old version of $\circ$, defined in the absence of $*$-structures, always turns Hopf algebras into Hopf algebras and implements a contravariant functor on the category of Hopf algebras that is self-adjoint on the right; see e.g. \cite[$\S$6.2]{MR0252485} for the first claim, and we leave the second one as an exercise. The analogue of this goes through for CQG algebras. We unpack this below. 

Note that any CQG algebra $A$ is naturally both a $*$-algebra (by definition) and a $*$-coalgebra. In fact, we can make it into a $*$-coalgebra in two ways, as both $S*$ and $*S$ are comultiplication-reversing, conjugate-linear involutions. We choose the latter structure: The involution making $A$ into a $*$-coalgebra in the sequel will be $a\mapsto (Sa)^*$. 

Next we observe that for every CQG algebra $A$, the $*$-coalgebra $A^\circ$ is again a CQG algebra. Indeed, one first argues that it is a Hopf algebra as in \cite[$\S$6.2]{MR0252485}. The $*$-algebra structure is given by
\[
	f^*(a)=f((Sa)^*)^*,\quad a\in A,\ f\in A^\circ,
\]
i.e. it is obtained from the $*$-coalgebra structure $a\mapsto (Sa)^*$ of $A$ via \Cref{eq.dual_functional}. It is an easy check now that the $*$-algebra and $*$-coalgebra structures on $A^\circ$ are compatible with the antipode 
\[
	(Sf)(a) = f(Sa),\quad a\in A,\ f\in A^\circ
\]
in precisely the right way: The $*$-coalgebra involution is $f\mapsto (Sf)^*$ for $f\in A^\circ$. Finally, we have to argue that the comodules of $A^\circ$ are unitarizable. We won't do this in any detail, but the idea is that the category of $A^\circ$-comodules is equivalent to that of $*$-representations of $A$ on finite-dimensional Hilbert spaces. In other words, $A^\circ$ will be exactly the CQG algebra reconstructed from this category by the general procedure described in \cite[Theorem 1.3]{MR943923} (or rather a minor modification thereof).

\begin{definition}\label{def.Hopf_dual}
For a CQG algebra $A$ we refer to $A^\circ$ endowed with the structures described above as the \define{Hopf dual} or \define{CQG dual} of $A$. 
\end{definition}

I claim further that as previewed above, the contravariant functor $A\mapsto A^\circ$ on the category of CQG algebras is self-adoint on the right, i.e. we have bijections 
\[
	\text{CQG morphisms }A\to B^\circ\quad \cong\quad \text{CQG morphisms }B\to A^\circ,
\] 
natural in $A$ and $B$ in the obvious sense. This follows from the fact that both sides of this expression correspond bijectively to pairings $\langle\cdot,\cdot\rangle:A\otimes B\to \bC$ that respect all the structure: 
\begin{itemize}
	\item The multiplication and comultiplication of the two Hopf algebras, via
				\[
					\langle x,uv\rangle = \langle x_1,u\rangle\langle x_2,v\rangle,\quad \forall x\in A,\quad \forall u,v\in B
				\]
				and
				\[
					\langle xy,u\rangle = \langle x,u_1\rangle\langle y,u_2\rangle,\quad \forall x,y\in A,\quad \forall u\in B.
				\]
	\item Units and counits by
				\[
					\langle x,1_B\rangle = \varepsilon_A(x) \text{ and } \langle 1_A,u\rangle = \varepsilon_B(u),\quad \forall x\in A,\quad \forall u\in B.
				\]
	\item The antipode through
				\[
					\langle S_Ax,u\rangle = \langle x,S_Bu\rangle,\quad \forall x\in A,\quad \forall u\in B,  
				\]
	\item and the $*$-structures by
				\[
					\langle x^*,(S_Bu)^*\rangle = \langle (S_Ax)^*,u^*\rangle = \langle x,u\rangle^*,\quad \forall x\in A,\quad \forall u\in B,
				\]where the very last `$*$' means complex conjugation.  
\end{itemize}

Finally, the self-adjunction of $\circ$ means that the canonical $*$-algebra map $A\to A^{\circ *}$, which makes sense for all $*$-algebras, actually factors through a CQG morphism $A\to A^{\circ\circ}$. The next result now follows from \Cref{pr.RFD_quot_alg}.

\begin{proposition}\label{pr.RFD_quot_CQG}
For any $*$-algebra $A$, the RFD quotient $\ol{A}$ is the image of the canonical morphism $A\to A^{\circ \circ}$ of CQG algebras. \qedhere
\end{proposition}

In particular, the RFD quotient is actually a quotient CQG algebra of $A$. This means that the coalgebra structure of $A$ is relevant to determining whether or not $A$ is RFD. Before recording this as a statement, we introduce some terminology.

\begin{definition}\label{def.jointly_full}
Let $\cC$ and $\cC_i$, $i\in I$ be concrete categories, i.e. such that objects are sets and morphisms are set maps via faithful functors from $\cC$ and $\cC_i$ to $\cat{Set}$. 

A family of functors $F_i:\cC\to\cC_i$ that preserve underlying sets and maps is said to be \define{jointly full} if any map of sets $f:S\to T$ which is in the image of all underlying functors $\cC_i\to\cat{Set}$ is in the image of $\cC\to\cat{Set}$.
\end{definition}

\begin{corollary}\label{cor.RFD_quot_CQG}
A CQG algebra $A$ is RFD if and only if there is a family of RFD quotient CQG algebras $A\to A_i$ such that the scalar corestriction functors $\cM^A\to\cM^{A_i}$ form a jointly full family. 
\end{corollary} 
\begin{proof}
One direction is immediate: If $A$ is RFD, then we can take the identity $A\to A$ itself for our RFD quotient CQG algebra. 

Conversely, assume there is a family of $A_i$'s as in the statement. I claim that the corestriction functor $\cM^A\to\cM^{\ol{A}}$ via the RFD quotient $A\to\ol{A}$ is full. Assuming this for a moment, the concluson follows from the fact that a map of complex cosemisimple coalgebras (such as $A\to\ol{A}$) is one-to-one if and only if the corresponding corestriction functor is full. Indeed, fullness of the corestriction functor is clearly equivalent to non-isomorphic simple $A$-comodules being sent to non-isomorphic simple $\ol{A}$-comodules; writing a cosemisimple coalgebra as a direct sum of coefficient subcoalgebras for its simple comodules finishes the argument. 

To prove the fullness claim, note first that since $A\to A_i$ are RFD quotients, they all factor through the universal RFD quotient $A\to\ol{A}$. But then for any two comodules $V,W\in\cM^A$, any $\ol{A}$-comodule map $f:V\to W$ is in particular an $A_i$-comodule map for every $i$. The joint fullness hypothesis implies that it is also an $A$-comodule map, and we are done.   
\end{proof}

\Cref{cor.RFD_quot_CQG} will be the main tool in the next section. Before we end this one, a word about why we had to settle for $A_u(n)$ and $B_u(n)$ rather than, for instance, the more general CQG algebras $A_u(Q)$ and $B_u(Q)$ from \Cref{ex.free,ex.ortho}. 

The issue is that for a CQG algebra to be RFD it must be what's usally referred to as \define{Kac type}: the antipode is automatically involutive. This is essentially \cite[Corollary A.3]{MR2210362}. That result deals with $C^*$-algebraic compact quantum groups rather than CQG algebras, but an RFD algebra always embeds in an RFD $C^*$-agebraic compact quantum group, to which we can then apply the cited corollary.

\section{Universal CQG algebras are residually finite-dimensional}\label{se.main}

The main result of the paper is

\begin{theorem}\label{th.main}
For any positive integer $n\ge 2$ that is different from $3$, the CQG algebras $A_u(n)$ and $B_u(n)$ are RFD in the sense of \Cref{def.RFD}. 
\end{theorem}

\begin{remark}
The statement also holds for $n=1$, but in that case it is immediate. We assume $n\ge 2$ to avoid having to deal with trivial exceptions to various arguments. I believe the result holds for $n=3$ as well, but the proof below does not cover that case; a tweak will likely do the trick, but I have been unable to find one so far.
\end{remark}

To simplify life somewhat, let's first reduce the problem to proving the RFD property for only one of the two algebras. To this end, we need to know a little about how $A=A_u(n)$ and $B=B_u(n)$ relate to one another. 

Consider the algebra $A'\le A$ generated by the elements $u^*_{ij}u_{kl}$. Similarly, let $B'\le B$ be the subalgebra generated by $v_{ij}v_{kl}$. The first auxiliary result is as follows.

\begin{proposition}\label{pr.free_vs_ortho}
There is a CQG algebra isomorphism $A'\cong B'$ defined by $u^*_{ij}u_{kl}\mapsto v_{ij}v_{kl}$.
\end{proposition} 
\begin{proof}
Give $C=\bC[t,t^{-1}]$ the CQG structure making $t$ a unitary grouplike element (i.e. $\Delta(t)=t\otimes t$, $\varepsilon(t)=1$ and $tt^*=1$); it is the one coming from realizing the algebra as representative functions of the circle group. Consider the coproduct CQG algebra $C*B$. 

It is easy to see that $tv_{ij}\in C*B$ satisfy the same relations as the $u_{ij}$, i.e. the matrices $(tv_{ij})_{i,j}$ and $(tv_{ji})_{i,j}$ are unitaries in $M_n(C*B)$. This means that there is a CQG algebra map $A\to C*B$ sending $u_{ij}$ to $tv_{ij}$, and it clearly restricts to $u^*_{ij}u_{kl}\mapsto v_{ij}v_{kl}$. 

So there is indeed a well defined CQG algebra map $A'\to B'$ as in the statement. One the one hand it is surjective because the generators $v_{ij}v_{kl}$ of $B'$ are in its image. On the other, it is one-to-one because $A\to C*B$ is (\cite[Theorem 3.4 (1)]{MR2468039}). This finishes the proof.  
\end{proof}

If we knew that $A$ is RFD, then so would $B'\cong A'$, since the RFD property clearly passes over to subalgebras. To get to $B$, we have to somehow lift RFD-ness from the subalgebra $B'$. We tackle this next. 

First, recall the notion of central morphism of CQG algebras from
\cite[Definition 2.1]{MR3175029} (based on the concept introduced in
the proof of \cite[Proposition 4.5]{MR2504527}); it is the natural
definition obtained by dualizing that of homomorphism from a compact
group into the center of another: 

\begin{definition}
A morphism $\pi:H\to K$ of CQG algebras is \define{central} if
\[
 \begin{tikzpicture}[anchor=base,>=stealth]
   \path +(0,0) node (1) {$H$}  +(3,.5)   node (2) {$H\otimes H$}
   +(6,0) node (3) {$H\otimes K$} +(1.5,-.5) node (4) {$H\otimes H$}
   +(4.5,-.5) node (5) {$H\otimes H$}; 
   \draw[->] (1) to [bend left=10] node[pos=.5,auto] {$\scriptstyle
     \Delta$} (2);
   \draw[->] (2) to [bend left=10] node[pos=.5,auto] {$\scriptstyle
     \id\otimes\pi$} (3);
   \draw[->] (1) to [bend right=10] node[pos=.5,auto,swap] {$\scriptstyle
     \Delta$} (4);
   \draw[->] (4) to [bend right=10] node[pos=.5,auto,swap] {$\scriptstyle\tau$} (5);
   \draw[->] (5) to [bend right=10] node[pos=.5,auto,swap]
   {$\scriptstyle\id\otimes \pi$} (3);
 \end{tikzpicture}
\]
commutes, where $\tau$ is the permutation of tensorands. 
\end{definition}

We henceforth only consider central maps which are also onto, and so
suppress the adjective `onto'. Recall also from \cite[\S 1.2]{MR3175029} that for a central map $\pi:H\to K$ of CQG algebras, one defines the third term $P\to H$ of an ``exact sequence'' $P\to H\to K$ as 
\[
	P=\{h\in H\ |\ \pi(h_1)\otimes h_2=1_K\otimes h\}. 
\]This is a symmetric definition (so the $\pi$ could have been applied
to the right hand tensorand instead), and $P$ is a Hopf subalgebra of
$H$ (called the Hopf kernel of $\pi$; it is the object denoted by
$\mathrm{HKer}(\pi)$ in \cite[Section 1]{MR1334152}). Moreover, $H\to K$ can be identified with the surjection $H\to H/HP^+$, where $P^+=\ker(\varepsilon|_P)$. In general, for a central map $H\to K$, $K$ is automatically a group algebra. 

More generally, this all goes through for maps $H\to K$ satisfying a
weaker property than centrality (cf. \cite[DEfinition
1.1.5]{MR1334152} or \cite[Section 2]{MR2504527}):

\begin{definition}\label{def.normal}
A map $\pi:H\to K$ of CQG algebras is \define{normal} if the sets 
\[
 \{h\in H\ |\ \pi(h_1)\otimes h_2=1_K\otimes h\}
\]
and 
\[
 \{h\in H\ |\ h_1\otimes\pi(h_2)=h\otimes 1_K\}
\]
are equal. 
\end{definition}

It is now easy to see that the center of the compact quantum group associated to $B$ (i.e. the largest central CQG quotient of $B$) is the map $B\to\bC[\bZ/2\bZ]$ defined by 
\begin{equation}\label{eq.centerB}
	v_{ij}\mapsto\delta_{ij}t. 
\end{equation}
Here, $t\in\bZ/2\bZ$ is the generator, and $\delta_{ij}$ is the Kronecker delta. The cocenter of $B$ is exactly $B'$. In other words: The coalgebra map \Cref{eq.centerB} induces a $\bC[\bZ/2\bZ]$-comodule structure on $B$, i.e. a $(\bZ/2\bZ)$-grading, and $B'\le B$ is the degree zero component. Hence the next result is relevant to lifting the RFD property from $B'$ to $B$.

\begin{theorem}\label{th.lift_RFD}
Let $H$ be a CQG algebra, $\Gamma$ a finite group, and $\pi:H\to \bC[\Gamma]$ a normal map as in \Cref{def.normal}. Denote by $P\le H$ the Hopf kernel of $\pi$. Then, $H$ is RFD if and only if $P$ is.  
\end{theorem}
\begin{proof}
As noted before, the RFD property for the large algebra $H$ implies it for the smaller algebra $P$, so it is the other implication that will be more interesting. We henceforth assume $P$ to be RFD. 

First, since $\Gamma$ is finite, the quotient $H\to\bC[\Gamma]$ factors through a normal map $\ol{H}\to\bC[\Gamma]$, where $H\to\ol{H}$ is the RFD quotient. Suppose we show that the composition $P\to H\to\ol{H}$ is one-to-one. We then get the diagram
\[
 \begin{tikzpicture}[anchor=base,>=stealth]
   \path +(0,-1) node (1) {$P$}  +(4,-1)   node (3) {$\bC[\Gamma]$,} +(2,-1.5) node (4) {$\ol{H}$}; 
   \draw[->] (1) to [bend right=10] (4);
   \node (2) at (2,-.5) {$H$}
   	edge[->,bend left=10] node[pos=.5,auto] {$\scriptstyle \pi$} (3);
   \draw[->] (2) -- (4);
   \draw[->] (1) to [bend left=10] (2);
   \draw[->] (4) to [bend right=10] (3);
 \end{tikzpicture}
\]
where the two horizontal rows are both exact in the sense that for both $H$ and $\ol{H}$ the CQG subalgebra $P$ is the Hopf kernel of the surjection onto $\bC[\Gamma]$. It now follows from \cite[Theorem 3.4]{MR3011789} that the vertical arrow is an isomorphism.   

It remains to prove that the map $P\to H\to\ol{H}$ is indeed injective. Since we are assuming that $P$ is RFD, this means showing that every finite-dimensional $*$-representation of $P$ on a Hilbert space embeds in (the restriction to $P$) of one of $H$. So let $M$ be a finite-dimensional Hilbert endowed with a $*$-action by $P$, and consider the $H$-module $H\otimes_PM$. The plan is to show that it is finite-dimensional and that it has a Hilbert space structure respecting the $*$-structure of $H$. We do these two things in reverse order. 

First, since $P\le H$ is an inclusion of cosemisimple Hopf algebras, it has a canonical retraction $H\to P$. Indeed, writing $H$ as a direct sum of matrix coalgebras corresponding to the simple $H$-comodules, $P$ is a direct sum of some of those coalgebras. This realizes $P$ as a direct summand of $H$, and the map $E$ is the projection induced by this direct sum decomposition. The so-called \define{expectation} $E$ intertwines the $*$ operations of $H$ and $P$, and is also a $P$-bimodule map. There is a general procedure of putting a pre-inner product on $H\otimes_PM$ for any Hilbert space $P$-representation $M$ in the presence of such an expectation: 
\[
	\langle h\otimes m,k\otimes n\rangle \stackrel{\text{def}}{=} E(h^*k)\langle m,n\rangle,\quad \forall h,k\in H,\quad \forall m,n\in M,
\]
where the right hand $\langle\cdot,\cdot\rangle$ is the inner product on $M$, assumed linear in the second variable (see e.g. \cite[Definition 1.3, Lemma 1.7]{MR0353003}). It is an easy check now that this plays well with respect to the $*$-structure of $H$, in the sense that 
\[
	\langle v,hw\rangle = \langle h^*v,w\rangle,\quad \forall h\in H,\quad \forall v,w\in H\otimes_PM. 
\]
Note further that the canonical $P$-module map $M\to H\otimes_PM$ sending $m$ to $1_H\otimes m$ is one-to-one since it has $E\otimes\id_M$ as a retraction. 

Finally, we need to show that $H\otimes_PM$ is finite-dimensional. Note that $H\otimes_PM$ is not only an $H$-module, but a $\Gamma$-graded one: The surjection $\pi:H\to\bC[\Gamma]$ induces a $\Gamma$-grading on $H$ such that $P$ is precisely its homogeneous component of degree $1_\Gamma\in\Gamma$ (because $P$ is the Hopf kernel of $\pi$).

Since $\pi$ is a surjection of cosemisimple coalgebras, it makes $H$ into a faithfully coflat comodule over $\bC[\Gamma]$. In the presence of this technical condition, a standard descent result (\cite[Theorem 2]{MR549940}) implies that $H\otimes_P$ implements an equivalence between the category $_P\cM$ of $P$-modules and that of $\Gamma$-graded $H$-modules, denoted by $_H^\Gamma\cM$. The inverse functor is $V\mapsto V_1$, where $V=\bigoplus_{\gamma\in\Gamma}V_\gamma$ is the grading of $V\in{_H^\Gamma}\cM$. 

The group $\Gamma$ acts on the category $_P\cM\simeq {_H^\Gamma}\cM$ by degree shift, with the autoequivalence implemented by $\eta\in\Gamma$ being defined by 
\[
	V\mapsto V^\eta,\ (V^\eta)_\gamma = V_{\gamma\eta},\quad\forall\ V\in{_H^\Gamma}\cM,\quad \forall \gamma\in\Gamma.
\]

The monoidal structure of $_P\cM$ can be transported via $H\otimes_P$ over to $_H^\Gamma\cM$: For $V,W\in{_H^\Gamma}\cM$ and $\gamma\in G$ we have $(V\otimes W)_\gamma = V_\gamma\otimes W_\gamma$. This description makes it clear that the action of $\Gamma$ on $_P\cM\simeq{_H^\Gamma}\cM$ from the previous paragraph is by monoidal autoequivalences. 

Since the finite-dimensional $P$-modules are exactly those that are rigid with respect to the monoidal structure in $_P\cM$, finite-dimensionality is preserved by any monoidal autoequivalence. In particular, the image $(H\otimes_PM)_\gamma\in{_P}\cM$ of $M\in{_P}\cM$ through the monoidal autoequivalence $\gamma\in\Gamma$ is finite-dimensional. But then the homogeneous components of $H\otimes_PM$ are finite-dimensional, and there are only finitely many components because $\Gamma$ is finite.  
\end{proof}

\begin{remark}
\Cref{th.lift_RFD} is very similar in spirit to \cite[Theorems 4.1 and 5.7]{MR2681256}. 
\end{remark}

We now have

\begin{proposition}\label{pr.A<->B}
For any $n\ge 2$, $A_u(n)$ is RFD if and only if $B_u(n)$ is.   
\end{proposition}
\begin{proof}
The preceding discussion, via \Cref{th.lift_RFD}, builds up to the implication ($A$ is RFD $\Rightarrow$ $B$ is RFD). Conversely, we noted in the proof of \Cref{pr.free_vs_ortho} that $A$ embeds in $\bC[t,t^{-1}]*B$, which proves the opposite implication using \Cref{pr.RFD_free_prod} again. 
\end{proof}

We can now focus on the algebras $A_u(n)$, whose RFD-ness takes up the rest of this section. The proof is by induction on $n$, by passing from $n=2$ to $n=4$ and then from any $n\ge 3$ to $n+1$. It seems somewhat more subtle to pass from $2$ to $3$, which is why $n=3$ is missing in \Cref{th.main}.

\begin{lemma}\label{le.n=2}
The CQG algebra $A_u(2)$ is RFD. 
\end{lemma} 
\begin{proof}
Consider the algbra $C(SU(2))$ of representative functions on $SU(2)$, i.e. the linear spans of matrix coefficients of finite-dimensional $SU(2)$-representations. It is generated as a $*$-algebra (in fact as an algebra) by the coefficients $w_{ij}$, $1\le i,j\le 2$ of the $2$-dimensinal vector representation. 

Just as in the proof of \Cref{pr.free_vs_ortho}, we have a CQG algebra morphism $A_u(2)\to \bC[t,t^{-1}]*C(SU(2))$ defined by $u_{ij}\mapsto tw_{ij}$. It follows from \cite[Lemme 7]{MR1484551} that this map is one-to-one, and hence it suffices to show that $\bC[t,t^{-1}]*C(SU(2))$. Since $\bC[t,t^{-1}]$ and $C(SU(2))$ are commutative CQG algebras, they are both RFD by \Cref{rem.comm_RFD}. But then the conclusion follows from \Cref{pr.RFD_free_prod}. 
\end{proof}

The inductive step in the proof of \Cref{th.main} splits in two, as indicated above we first pass from $n=2$ to $n=4$, and then from $n-1$ to $n$ for $n\ge 5$. In both proofs we make use of \Cref{cor.RFD_quot_CQG}.  This latter result says that it suffices to find CQG algebra morphisms from $A=A_u(n)$ into some RFD CQG algebras $A_i$ such that the induced functors $\cM^A\to\cM^{A_i}$ form a jointly full family. In fact, we will use only two $A_i$'s, which we now proceed to describe. 

In both proofs, one RFD quotient of $A$ is the abelianization $A\to C(U(n))$ obtained by imposing the commutativity between the generators $u_{ij}$ of $A$. The resulting quotient is, just as the notation suggests, the algebra of representative functions on the $n\times n$ unitary group; the images $w_{ij}$ of $u_{ij}$ are the coefficients of the standard $n$-dimensional representation of $U(n)$.

For passing from $n=2$ to $n=4$, the other RFD quotient of $A$ that we consider mods out the Hopf ideal generated by $u_{ij}$ for $i=1,2$ and $j=3,4$ or $i=3,4$ and $j=1,2$. In other words, we break up the $4\times 4$ matrix into $2\times 2$ blocks along the diagonal, and kill off the remaining $u$'s. The quotient is $D=A_u(2)*A_u(2)$.

Similarly, for $n\ge 5$, the second quotient of $A$ that we consider is the one by the Hopf ideal generated by $u_{in}$ and $u_{nj}$ for $1\le i,j\le n-1$. This time around we break up the $n\times n$ matrix into an $(n-1)\times (n-1)$ upper left hand block $(u_{ij})$, $1\le i,j\le n-1$ and a lower right hand corner $u_{nn}$, and we annihilate the other $u$'s. Denoting $C=\bC[t,t^{-1}]$, the resulting quotient is nothing but $C*A_u(n-1)$, with the $u_{ij}\in A$, $1\le i,j\le n-1$ mapping onto the corresponding generators $x_{ij}$ of $A_u(n-1)$, and $u_{nn}$ being sent to $t\in C$. We denote this quotient again by $D$, so that the same letter denotes two different algebras, depending on context. 

Denote by $V$ the $n$-dimensional $A$-comodule whose corresponding matrix coalgebra is spanned by $u_{ij}$. As in \Cref{ex.free}, we denote by $(e_i)_{i=1}^n$ a basis such that the comodule structure on $V$ is $e_j\mapsto \sum_i e_i\otimes u_{ij}$. The inner product making the $e_i$'s orthonormal is compatible with this comodule structure in the appropriate sense, and it is the one we use whenever thinking of $V$ as a Hilbert space. We denote the dual basis in $V^*$ by $(f_i)_{i=1}^n$. 

Joint fullness of the two functors out of $\cM^A$ requires that for any two finite-dimensional $A$-comodules $W,W'$, any map $W\to W'$ compatible both with the $U(n)$-representation structures and the $D$-comodule structures is automatically an $A$-comodule morphism. Identifying $\mathrm{Hom}(W,W')$ with $W'\otimes W^*$, this means showing that all elements of $W'\otimes W^*$ fixed by both $U(n)$ and $D$ are fixed by $A$ (an element $w$ of a comodule over a Hopf algebra is fixed or invariant if the comodule structure map acts on it by $w\mapsto w\otimes 1$).  

Because $A$ is generated by $u_{ij}$ and $u_{ij}^*$ as an algebra, every $A$-comodule is a subcomodule of a finite direct sum of tensor products $V^{(\varepsilon_i)}=V^{\varepsilon_1}\otimes\cdots\otimes V^{\varepsilon_\ell}$, where $\varepsilon_i$ are either blank or `$*$'. Hence, it suffices to substitute such tensor products $V^{(\varepsilon_i)}$ for $W'\otimes W^*$ in the previous paragraph, and show that elements of $V^{(\varepsilon_i)}$ fixed by $U(n)$ and $D$ are fixed by $A$. 

Let us now recall some facts about the category $\cM^A$, to get a better grasp of what the above-stated goal entails. According to \cite[Th\'eor\`eme 1]{MR1484551} (and as recounted in \Cref{ex.free}), the simples are indexed by words in $V$ and $V^*$. If $a_x$ denotes the irreducible corresponding to the word $x$, then the decomposition of tensor products in $\cM^A$ is given by 
\begin{equation}\label{eq.free_bis}
	a_xa_y = \sum_{x=vg,y=g^*w} a_{vw}.
\end{equation} 
Here, recall that expressions such as $vg$ stand for concatenation of words $v$ and $g$, and $g\mapsto g^*$ is the anti-multiplicative involution on the free semiring on $V$ and $V^*$ that interchanges $V$ and $V^*$. 

Unpacking this at the level of comodules, we see that homomorphisms from a tensor product $V^{(\varepsilon_i)}$ to the trivial comodule are of the following form: One considers all ways of pairing off a $V$ with a $V^*$ among the $V^{\varepsilon_i}$'s in such a way that non-intersecting segments can be drawn to connect the pairs. The space $^A\mathrm{Hom}(V^{(\varepsilon_i)},\bC)$ of coinvariants is the span of such pairings, which we refer to as \define{non-crossing}. The invariants $^A\mathrm{Hom}(\bC,V^{(\varepsilon_i)})$ can be identified with the coinvariants, since all comodules in sight have we have compatible inner products; for this reason, we do not distinguish between invariants and coinvariants, which might makes for a certain amount of hopefully non-confusing sloppiness in the language. 

Here are some examples of coinvariants in $\cM^A$, with white and black dots standing for $V$ and $V^*$ respectively. The left hand side depicts $(V\otimes V^*)^{\otimes 3}$, while the right hand side is $(V^*)^{\otimes 2}\otimes V^{\otimes 2}$.

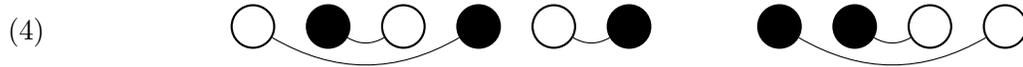
\begin{equation}\label{eq.A_coinv}
	\begin{tikzpicture}[baseline=(current  bounding  box.center),
			wh/.style={circle,draw=black,thick,inner sep=2mm},
			bl/.style={circle,draw=black,fill=black,thick,inner sep=2mm}]
		\node (1) at (0,0) [wh] {};
		\node (2) at (1,0) [bl] {};
		\node (3) at (2,0) [wh] {};
		\node (4) at (3,0) [bl] {};
		\node (5) at (4,0) [wh] {};
		\node (6) at (5,0) [bl] {};
		\draw (2) to [bend right] (3);
		\draw (1) to [bend right] (4);
		\draw (5) to [bend right] (6);
		\node (1') at (7,0) [bl] {};
		\node (2') at (8,0) [bl] {};
		\node (3') at (9,0) [wh] {};
		\node (4') at (10,0) [wh] {};
		\draw (1') to [bend right] (4');
		\draw (2') to [bend right] (3'); 	
	\end{tikzpicture}
\end{equation}
By comparison, Schur-Weyl duality says that the coinvariants for the $U(n)$-action on $V^{(\varepsilon_i)}$ are the span of \define{all} pairings between a $V$ and a $V^*$. Some examples for the same two comodules as in the previous picture: 
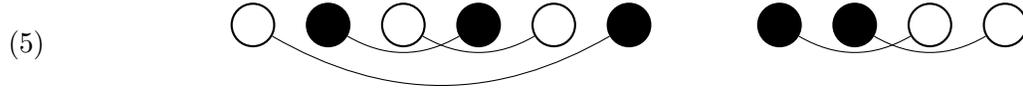
\begin{equation}\label{eq.U_coinv}
	\begin{tikzpicture}[baseline=(current  bounding  box.center),
			wh/.style={circle,draw=black,thick,inner sep=2mm},
			bl/.style={circle,draw=black,fill=black,thick,inner sep=2mm}]
		\node (1) at (0,0) [wh] {};
		\node (2) at (1,0) [bl] {};
		\node (3) at (2,0) [wh] {};
		\node (4) at (3,0) [bl] {};
		\node (5) at (4,0) [wh] {};
		\node (6) at (5,0) [bl] {};
		\draw (2) to [bend right] (4);
		\draw (1) to [bend right] (6);
		\draw (3) to [bend right] (5);
		\node (1') at (7,0) [bl] {};
		\node (2') at (8,0) [bl] {};
		\node (3') at (9,0) [wh] {};
		\node (4') at (10,0) [wh] {};
		\draw (1') to [bend right] (3');
		\draw (2') to [bend right] (4'); 	
	\end{tikzpicture}
\end{equation}

When $n=2$ and hence $D$ is $A_u(2)*A_u(2)$, the comodule $V$ breaks up as the direct sum between the span $W$ of $e_1,e_2$ and the span $U$ of $e_3,e_4$. The dual $V^*$ breaks up accordingly as $W^*\oplus U^*$. The typical tensor product $V^{(\varepsilon_i)}$, $1\le i\le\ell$ decomposes into $2^\ell$ summands, according to whether we choose $W^{\varepsilon_i}$ or $U^{\varepsilon_i}$ for each $i$. 

To get coinvariants we have to pair $W$ tensorands with $W^*$'s and $U$'s with $U^*$'s, but just as for $A$, in a non-crossing manner. This follows from the description of coinvariants for $A_u(2)$ and the fact that the simple comodules for a coproduct $H_1*H_2$ of CQG algebras are precisely the tensor products $V_1\otimes V_2\otimes\cdots\otimes V_\ell$, where $V_i$ are simple comodules over $H_1$ or $H_2$ in an alternating fashion and no $V_i$'s are trivial (see e.g. \cite[Theorem 3.10]{MR1316765}). 

The same discussion applies for $n\ge 5$ with $W$ being the span of $e_i$, $1\le i\le n-1$ and $U$ being the one-dimensional vector space spanned by $e_n$. Once more, we use the same symbols and rely on context to differentiate between the two situations described in this paragraph and the previous one. 
 
In the pictures below, small white (black) circles represent $U$'s (respectively $U^*$'s), and similarly, the large circles are $W$'s. The left hand side is a coinvariant in the summand $W\otimes U^*\otimes W\otimes W^*\otimes U\otimes W^*$ of $(V\otimes V^*)^{\otimes 3}$. The upper right hand side question mark indicates that there are no non-zero coinvariants for the $D$-comodule $U^*\otimes W^*\otimes U\otimes W$, because the only way of pairing $U$ to $U^*$ and $W$ to $W^*$ is not non-crossing. On the other hand, $W^*\otimes U^*\otimes U\otimes W\cong (U\otimes W)^*\otimes(U\otimes W)$ does have the obvious coinvariant pairing off $U\otimes W$ with its dual. 
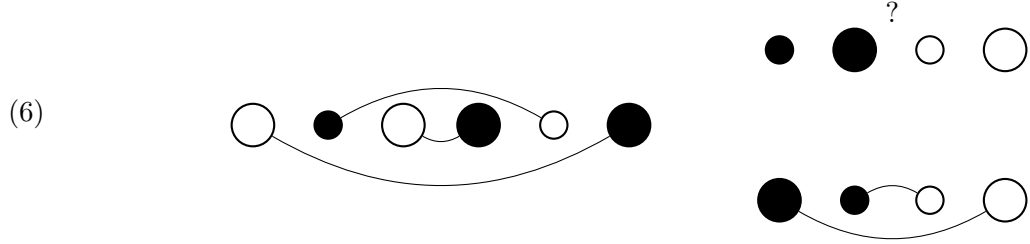
\begin{equation}\label{eq.D_coinv}
	\begin{tikzpicture}[baseline=(current  bounding  box.center),
			wh/.style={circle,draw=black,thick,inner sep=2mm},
			bl/.style={circle,draw=black,fill=black,thick,inner sep=2mm},
			whs/.style={circle,draw=black,thick},
			bls/.style={circle,draw=black,fill=black,thick}]
		\node (1) at (0,-1) [wh] {};
		\node (2) at (1,-1) [bls] {};
		\node (3) at (2,-1) [wh] {};
		\node (4) at (3,-1) [bl] {};
		\node (5) at (4,-1) [whs] {};
		\node (6) at (5,-1) [bl] {};		\draw (1) to [bend right] (6);
		\draw (2) to [bend left] (5);
		\draw (3) to [bend right] (4);
		\node (1') at (7,0) [bls] {};
		\node (2') at (8,0) [bl] {};
		\node (3') at (9,0) [whs] {};
		\node (4') at (10,0) [wh] {};
		\node (5') at (8.5,.5) {$?$};
		\node (1'') at (7,-2)  [bl]   {};
		\node (2'') at (8,-2)	 [bls]  {};
		\node (3'') at (9,-2)	 [whs]  {};
		\node (4'') at (10,-2) [wh]   {};
		\draw (1'') to [bend right] (4'');
		\draw (2'') to [bend left] (3''); 
	\end{tikzpicture}\end{equation}
Our goal of showing joint fullness now consists of proving that for any $V^{(\varepsilon_i)}$, a coinvariant that is both in the span of pictures \Cref{eq.U_coinv} and that of \Cref{eq.D_coinv} must necessarily be in the span of \Cref{eq.A_coinv}.

\begin{remark}\label{rem.lin_ind}
It will be important in the sequel to note that although in general the $U(n)$-pairings are not linearly independent, the non-crossing pairings are whenever $n\ge 2$. So the non-crossing pairing pictures form a basis for the space of coinvariants of any $A$-comodule $V^{(\varepsilon_i)}$. 
\end{remark}

Note that if $V^{(\varepsilon_i)}$ for is to have any non-zero $U(n)$-coinvariants at all for a sequence $(\varepsilon_1,\ldots,\varepsilon_\ell)$ of blanks and $*$'s, then $\ell$ must be even ($2k$, say), and there must be an equal number of blanks and $*$'s (i.e. an equal number of $V$'s and $V^*$'s). 

Now, since $V^{(\varepsilon_i)}$ is isomorphic to $V^{\otimes k}\otimes(V^*)^{\otimes k}$ as a vector space, coinvariants and invariants can be identified with endomorphisms of $V^{\otimes k}$. Since we only work with one $V^{(\varepsilon_i)}$ at a time, we can fix this identification once and for all once we decide in which order to pull the $V^*$'s out to the left as $V$'s. We will assume such an identification has been made whenever convenient, and we freely change points of view to talk about the tensors we are manipulating as either elements of $V^{(\varepsilon_i)}$, or homomorphisms from it to $\bC$, or finally, endomorphisms of $V^{\otimes k}$. We often refer to $U(n)$-coinvariants as permutations, for instance, because every picture \Cref{eq.U_coinv} corresponds to a permutation of the $k$ tensorands once this identification has been made. Correspondingly, we may refer to pictures \Cref{eq.A_coinv} as \define{non-crossing permutations}. 
 
With all of this in place, we can forge on towards the proof of the main theorem.

\begin{lemma}\label{le.2->4}
Let $n=4$, so that $D=A_u(2)*A_u(2)$. For any choice of symbols $(\varepsilon_1,\ldots,\varepsilon_{2k})$ consisting of $k$ blanks and $k$ $*$'s, a map $V^{(\varepsilon_i)}\to\bC$ that is both a $U(n)$-coinvariant and a $D$-coinvariant is in the span of the non-crossing pairings \Cref{eq.A_coinv}. 
\end{lemma}
\begin{proof}
Let $S_k$ be the symmetric group on $k$ symbols. We think of coinvariants as linear combinations of permutations $\sigma\in S_k$ of the $k$ tensorands in $V^{\otimes k}$, as explained above. We start out with such a linear combination, say $f=\sum_{\sigma\in S_k} a_\sigma\sigma$ acting on $V^{\otimes k}$ which is also a $D$-coinvariant. When restricted to $W^{\otimes k}$, $f$ agrees with a linear combination of non-crossing pairings appropriate to $W^{(\varepsilon_i)}$. Subtracting the corresponding $A$-coinvariant of $V^{(\varepsilon_i)}$, we can assume that the restriction of $f$ to $W^{\otimes k}$ is zero.

So the new goal is: Show that if $f=\sum_{S_k} a_\sigma\sigma$ is a $D$-coinvariant vanishing on $W^{\otimes k}$, then $f=0$. Recall that we have decomposed $V^{\otimes k}$ into $2^k$ summands, according to a choice of either $W$ or $U$ in each of the $k$ positions in the tensor product. For some $1\le s\le k$, select one of the $\binom{n}{s}$ summands isomorphic to $U^{\otimes s}\otimes W^{\otimes(k-s)}$. We restrict our attention to it for the rest of the proof, and hence there is no ambiguity in the notation. 

Because $f$ is in the span of the $D$-coinvariants \Cref{eq.D_coinv}, its restriction to $U^{\otimes s}\otimes W^{\otimes(k-s)}$ acts as a linear combination $\sum b_\tau\tau$ of non-crossing permutations $\tau$. Moreover, because $f$ is a $U(4)$-intertwiner and hence a $GL(4)$-intertwiner, it acts as the same linear combination of permutations on $(U')^{\otimes s}\otimes W^{\otimes(k-s)}$ for \define{any} choice of complement $U'$ of $W$ in $V$. By continuity, we can ``fold'' $U'$ onto $W$ and conclude that $f$ acts as $\sum b_\tau\tau$ on $W^{\otimes k}$:
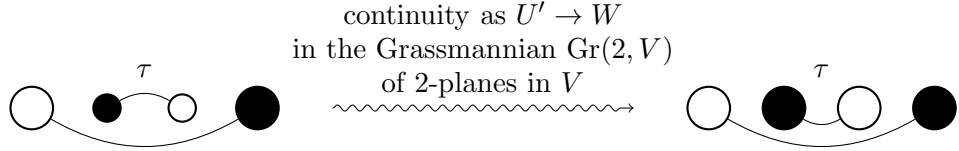
\[
	\begin{tikzpicture}[baseline=(current  bounding  box.center),
			wh/.style={circle,draw=black,thick,inner sep=2mm},
			bl/.style={circle,draw=black,fill=black,thick,inner sep=2mm},
			whs/.style={circle,draw=black,thick},
			bls/.style={circle,draw=black,fill=black,thick}]
		\node (1) at (0,0) [wh] {};
		\node (2) at (1,0) [bls] {};
		\node (3) at (2,0) [whs] {};
		\node (4) at (3,0) [bl] {};
		\node (5) at (1.5,.5) {$\tau$};
		\draw (1) to [bend right] (4);
		\draw (2) to [bend left] (3);
		\draw [->,decorate,decoration={snake,amplitude=.4mm,segment length=2mm,post length=1mm}]
		(4,0) -- (8,0)
		node [above,align=center, midway]
		{continuity as $U'\to W$\\ in the Grassmannian $\mathrm{Gr}(2,V)$\\ of $2$-planes in $V$};	
		\node (1') at (9,0) [wh] {};
		\node (2') at (10,0) [bl] {};
		\node (3') at (11,0) [wh] {};
		\node (4') at (12,0) [bl] {};
		\node (5') at (10.5,.5) {$\tau$};
		\draw (1') to [bend right] (4');
		\draw (2') to [bend right] (3');
	\end{tikzpicture}
\]
But we are assuming that $f|_{W^{\otimes k}}$ is zero, and hence $b_\tau$ are all zero by the linear independence of non-crossing permutations on $W^{\otimes k}$ (\Cref{rem.lin_ind}).  
\end{proof}

This bootstraps RFD-ness a little bit, from $n=2$ up to $n=4$. The rest is taken care of by its companion result:

\begin{lemma}\label{le.4->}
Let $n\ge 5$, so that $D=C*A_u(n-1)$. For any choice of symbols $(\varepsilon_1,\ldots,\varepsilon_{2k})$ consisting of $k$ blanks and $k$ $*$'s, a map $V^{(\varepsilon_i)}\to\bC$ that is both a $U(n)$-coinvariant and a $D$-coinvariant is in the span of the non-crossing pairings \Cref{eq.A_coinv}.
\end{lemma}
\begin{proof}
The argument is very similar to the proof of \Cref{le.2->4}; it is only the last step that requires more care. As before, we fix an endomorphism $f$ of $V^{\otimes k}$ which intertwines both the action of $U(n)$ and the coaction of $D$, and we assume that its restriction to $W^{\otimes k}$ is zero. We then seek to show that $f$ also vanishes on any summand $U^{\otimes s}\otimes W^{\otimes(k-s)}$. 

Once again this applies to any line $U'$ complementary to $W$, and by a limiting argument inside the projective space $\bP(V)$ we can assume, as in the previous proof, that $U'$ is contained in $W$. The restriction of $f$ to $U^{\otimes s}\otimes W^{\otimes(k-s)}$ acted as a linear combination $\sum b_\tau\tau$ of non-crossing permutations $\tau$, and so its restriction to $(U')^{\otimes s}\otimes W^{\otimes(k-s)}$ is this same linear combination, which in addition we know vanishes. 

The problem this time around is that $U'$ is one-dimensional. This means that $\sum b_\tau\tau$ only vanishes when restricted to those tensors in $W^{\otimes k}$ which are symmetric in the $s$ tensorands $(U')^{\otimes s}$, and hence now, unlike in the proof of \Cref{le.2->4}, we cannot conclude that all $b_\tau$ vanish. However, we do not need to. 

Collect the $\tau$'s into classes $C_\alpha$ based on which $s$ spots among $1,2,\ldots,k$ they send the $s$ tensorands from $(U')^{\otimes s}$. The index $\alpha$, in other words, ranges over the $s$-element subsets of $\{1,\ldots,k\}$. Choose a complement $W'$ of $U'$ in $W$. Decompose tensors in $W^{\otimes k}$ according to the splitting $W=U'\oplus W'$, we conclude that for each class $C_\alpha$ the linear combination $\Sigma_\alpha=\sum_{\tau\in C_\alpha}b_\tau\tau$ vanishes on $(U')^{\otimes s}\otimes(W')^{\otimes(k-s)}$. Each $\tau\in C_\alpha$ induces a non-crossing intertwiner $\ol{\tau}$ of $(W')^{\otimes(k-s)}$ by simply looking at what the permutation does to the $k-s$ tensorands from $W'$. Since $\dim(W')\ge 2$, the linear independence of non-crossing permutations (\Cref{rem.lin_ind}) implies that for any non-crossing permutation $\ol{\sigma}$ of $(W')^{\otimes(k-s)}$, the sum of all $b_\tau$ for $\tau\in C_\alpha$ such that $\ol{\tau}=\ol{\sigma}$ is zero. But then the restriction of $\Sigma_\alpha$ to $(U'')^{\otimes s}\otimes W^{\otimes(k-s)}$ (or indeed $(U'')^{\otimes s}\otimes V^{\otimes(k-s)}$) vanishes for \define{any} one-dimensional space $U''$, in particular for $U''=U$. This gives the desired conclusion that $f$ restricted to $U^{\otimes s}\otimes W^{\otimes(k-s)}$ is zero. 
\end{proof}

\begin{remark}
It is at the point where we chose a complement $W'$ of $U'$ in $W$ that $n\ge 5$ played a role. That condition implies that non-crossing permutations on the $(n-2)$-dimensional space $W'$ are linearly independent. We only need $n=4$ for this to work, and hence the proof would get the RFD property for $n=4$ if we had it for $n=3$; as noted before, we do not.  
\end{remark}

\begin{proof}[Proof of \Cref{th.main}]
This is now simply a matter of assembling everything together: \Cref{le.n=2} gets the induction going, then \Cref{le.2->4} pushes us up to $n=4$, and finally, \Cref{le.4->} gets the RFD property for $A_u(n)$ for every $n\ge 5$. Finally, from \Cref{pr.A<->B} we then deduce that $B_u(n)$ is RFD for the same values of $n$.  
\end{proof}

\begin{remark}
Note that \Cref{le.2->4,le.4->} can be construed as an alternate proof for the fusion rules \Cref{eq.free_bis} of $A_u(n)$ for $n\ge 4$. Indeed, the conjunction of the two lemmas shows that the coinvariants of $V^{(\varepsilon_i)}$ for these CQG algebras are spans of non-crossing pairings. Conversely, the non-crossing pairings are coinvariants of $V^{(\varepsilon_i)}$ for any comodule $V$ over any CQG algebra. 
\end{remark}

\subsection{On maximal almost periodic discrete quantum groups}

We now make the connection between residual finite-dimensionality as treated here and the notion of maximal almost periodicity from \cite{MR2210362}. 

One starts out by regarding the CQG algebra $B$ underlying a compact quantum group as the group algebra of a discrete quantum group. The object dual to a discrete quantum group will then be a kind of dual to $B$. This is typically phrased in the language of $C^*$-algebras: One first takes the direct sum $B^\bullet$ of the matrix algebras dual to the matrix subcoalgebras of $B$. This is a non-unital $*$-algebra, and can be $*$-represented on Hilbert spaces. Taking for every element $x\in B^\bullet$ the supremum of the norms achieved by $x$ in all of these representations endows $B^\bullet$ with a norm, and the completion with respect to this norm is a non-unital $C^*$-algebra $A$, which is to be thought of as the algebra of functions vanishing at infinity on the fictitious underlying discrete space of this quantum group. 

The quantum function algebra $A$ has something like a comultiplication $\Delta$ reflecting the fact that it is trying to be functions on a group, but it is something somewhat more sophisticated than in the purely algebraic situations we are dealing with in this paper. The map $\Delta_A$ lands inside $M(A\otimes A)$, where the tensor product stands for the completion of the algebraic tensor product with respect to the smallest possible $C^*$-norm on it, and $M(-)$ is the so-called \define{multiplier algebra} of a non-unital $C^*$-algebra. 

For a general $C^*$-algebra $D$, $M(D)$ is the largest unital $C^*$-algebra containing $D$ as an essential ideal (`essential' meaning that every non-zero closed ideal intersects $D$ non-trivially).  There is a natural topology on $M(A)$ with respect to which $A$ is dense, called the \define{strict} topology; we refer to Chapter 2 of \cite{MR1222415} for generalities on multiplier algebras.  

Within this framework (and in fact more generally), So\l tan introduces in \cite[2.14]{MR2210362} the notion of \define{quantum Bohr compactification} for the discrete quantum group in question. Classically, the Bohr compactification of a discrete group $\Gamma$ is a compact group mapped into from $\Gamma$ universally; the continuous functions on the Bohr compactification restrict to the so-called \define{almost periodic} functions on the initial, non-compact group. Dually, in the quantum case it consists of an appropriately universal comultiplication-preserving $C^*$-algebra inclusion of a $C^*$-completed CQG algebra $\mathbb{AP}(A)$ into the multiplier algebra $M(A)$. 

To get a better handle on this object in the context of this paper, let us note here that the CQG algebra underlying the $C^*$-algebraic compact quantum group $\mathbb{AP}(A)$ is precisely the Hopf dual $B^\circ$ from \Cref{def.Hopf_dual}. We need this below, in the proof of \Cref{pr.MAP}. 

Now, ordinary discrete grups are said to be \define{maximal almost periodic} if they possess enough almost periodic functions, i.e. if they embed in their Bohr compactifications. The dual version of this property, according to \cite[4.5]{MR2210362}, ought to be as follows:

\begin{definition}
The discrete quantum group with underlying function algebra $A$ is \define{maximal almost periodic} if the subalgebra $\mathbb{AP}(A)\le M(A)$ is dense with respect to the aforementioned strict topology. 
\end{definition}

Part (1) of \cite[Proposition 4.10]{MR2210362} shows that the discrete quantum group is indeed maximal almost periodic whenever the universal $C^*$-completion of the CQG algebra $B$ is RFD in the $C^*$ sense, i.e. the $C^*$ envelope has a separating family of (continuous) $*$-representations on finite-dimensional Hilbert spaces. 

An examination of the proof of that result shows that it goes through so long as \define{some} $C^*$ completion of $B$ is RFD. Equivalently, this means exactly that $B$ itself is RFD as a $*$-algebra. In conclusion, the following slightly more general statement holds:

\begin{proposition}\label{pr.MAP}
If the CQG algebra $B$ is RFD in the sense of \Cref{def.RFD}, then the discrete dual of the compact quantum group corresponding to $B$ is maximal almost periodic. \qedhere
\end{proposition}

As a direct consequence of \Cref{th.main} we then have

\begin{theorem}\label{th.MAP}
For $n\ge 2$ but different from $3$, the discrete quantum groups dual to $A_u(n)$ and $B_u(n)$ are maximal almost periodic.  \qedhere
\end{theorem}


\section*{Acknowledgements}

This work constitutes one chapter of the author's PhD dissertation at the University of California at Berkeley. I would like to thank my advisers Vera Serganova and Nicolai Reshetikhin for all of their support, as well as Piotr So\l tan for helpful conversations on the contents of \cite{MR2210362}.

The work was partially supported by the Danish National research Foundation through the QGM Center at Aarhus University and by the Chern-Simons Chair in Mathematical Physics at UC Berkeley.


\begin{thebibliography}{10}

\bibitem{MR1334152}
N.~Andruskiewitsch and J.~Devoto.
\newblock Extensions of {H}opf algebras.
\newblock {\em Algebra i Analiz}, 7(1):22--61, 1995.

\bibitem{MR2681256}
Nicol{\'a}s Andruskiewitsch and Julien Bichon.
\newblock Examples of inner linear {H}opf algebras.
\newblock {\em Rev. Un. Mat. Argentina}, 51(1):7--18, 2010.

\bibitem{MR1484551}
Teodor Banica.
\newblock Le groupe quantique compact libre {$\mathrm{U}(n)$}.
\newblock {\em Comm. Math. Phys.}, 190(1):143--172, 1997.

\bibitem{MR2468039}
Teodor Banica.
\newblock A note on free quantum groups.
\newblock {\em Ann. Math. Blaise Pascal}, 15(2):135--146, 2008.

\bibitem{MR2679923}
Teodor Banica and Julien Bichon.
\newblock Hopf images and inner faithful representations.
\newblock {\em Glasg. Math. J.}, 52(3):677--703, 2010.

\bibitem{MR3011789}
Teodor Banica, Julien Bichon, Beno{\^{\i}}t Collins, and Stephen Curran.
\newblock A maximality result for orthogonal quantum groups.
\newblock {\em Comm. Algebra}, 41(2):656--665, 2013.

\bibitem{MR2914541}
Teodor Banica, Uwe Franz, and Adam Skalski.
\newblock Idempotent states and the inner linearity property.
\newblock {\em Bull. Pol. Acad. Sci. Math.}, 60(2):123--132, 2012.

\bibitem{MR3175029}
Alexandru Chirvasitu.
\newblock Centers, cocenters and simple quantum groups.
\newblock {\em J. Pure Appl. Algebra}, 218(8):1418--1430, 2014.

\bibitem{MR590864}
Man~Duen Choi.
\newblock The full {$C^{\ast} $}-algebra of the free group on two generators.
\newblock {\em Pacific J. Math.}, 87(1):41--48, 1980.

\bibitem{MR1310296}
Mathijs~S. Dijkhuizen and Tom~H. Koornwinder.
\newblock C{QG} algebras: a direct algebraic approach to compact quantum
  groups.
\newblock {\em Lett. Math. Phys.}, 32(4):315--330, 1994.

\bibitem{MR1168356}
Ruy Exel and Terry~A. Loring.
\newblock Finite-dimensional representations of free product {$C^*$}-algebras.
\newblock {\em Internat. J. Math.}, 3(4):469--476, 1992.

\bibitem{MR1052340}
K.~R. Goodearl and P.~Menal.
\newblock Free and residually finite-dimensional {$C^*$}-algebras.
\newblock {\em J. Funct. Anal.}, 90(2):391--410, 1990.

\bibitem{MR1492989}
Anatoli Klimyk and Konrad Schm{\"u}dgen.
\newblock {\em Quantum groups and their representations}.
\newblock Texts and Monographs in Physics. Springer-Verlag, Berlin, 1997.

\bibitem{MR1741102}
Johan Kustermans and Lars Tuset.
\newblock A survey of {$C^*$}-algebraic quantum groups. {I}.
\newblock {\em Irish Math. Soc. Bull.}, (43):8--63, 1999.

\bibitem{MR1712872}
Saunders Mac~Lane.
\newblock {\em Categories for the working mathematician}, volume~5 of {\em
  Graduate Texts in Mathematics}.
\newblock Springer-Verlag, New York, second edition, 1998.

\bibitem{MR1243637}
Susan Montgomery.
\newblock {\em Hopf algebras and their actions on rings}, volume~82 of {\em
  CBMS Regional Conference Series in Mathematics}.
\newblock Published for the Conference Board of the Mathematical Sciences,
  Washington, DC, 1993.

\bibitem{MR2894855}
David~E. Radford.
\newblock {\em Hopf algebras}, volume~49 of {\em Series on Knots and
  Everything}.
\newblock World Scientific Publishing Co. Pte. Ltd., Hackensack, NJ, 2012.

\bibitem{MR0353003}
Marc~A. Rieffel.
\newblock Induced representations of {$C^{\ast} $}-algebras.
\newblock {\em Advances in Math.}, 13:176--257, 1974.

\bibitem{MR2210362}
Piotr~M. So{\l}tan.
\newblock Quantum {B}ohr compactification.
\newblock {\em Illinois J. Math.}, 49(4):1245--1270, 2005.

\bibitem{MR0252485}
Moss~E. Sweedler.
\newblock {\em Hopf algebras}.
\newblock Mathematics Lecture Note Series. W. A. Benjamin, Inc., New York,
  1969.

\bibitem{MR549940}
Mitsuhiro Takeuchi.
\newblock Relative {H}opf modules---equivalences and freeness criteria.
\newblock {\em J. Algebra}, 60(2):452--471, 1979.

\bibitem{MR1382726}
Alfons Van~Daele and Shuzhou Wang.
\newblock Universal quantum groups.
\newblock {\em Internat. J. Math.}, 7(2):255--263, 1996.

\bibitem{MR1316765}
Shuzhou Wang.
\newblock Free products of compact quantum groups.
\newblock {\em Comm. Math. Phys.}, 167(3):671--692, 1995.

\bibitem{MR2504527}
Shuzhou Wang.
\newblock Simple compact quantum groups. {I}.
\newblock {\em J. Funct. Anal.}, 256(10):3313--3341, 2009.

\bibitem{MR1222415}
N.~E. Wegge-Olsen.
\newblock {\em {$K$}-theory and {$C^*$}-algebras}.
\newblock Oxford Science Publications. The Clarendon Press, Oxford University
  Press, New York, 1993.
\newblock A friendly approach.

\bibitem{MR901157}
S.~L. Woronowicz.
\newblock Compact matrix pseudogroups.
\newblock {\em Comm. Math. Phys.}, 111(4):613--665, 1987.

\bibitem{MR943923}
S.~L. Woronowicz.
\newblock Tannaka-{K}re\u\i n duality for compact matrix pseudogroups.
  {T}wisted {$\mathrm{SU}(N)$} groups.
\newblock {\em Invent. Math.}, 93(1):35--76, 1988.

\end{thebibliography}
\bibliographystyle{plain}

\end{document}